\documentclass[reqno]{amsart}
\usepackage[vmargin=33truemm,hmargin=31truemm]{geometry}

\usepackage{xcolor}

\usepackage{listings}
\usepackage{mathtools}
\mathtoolsset{showonlyrefs, showmanualtags}
\usepackage{aligned-overset}

\usepackage{bm}
\usepackage{mleftright}
\usepackage{enumitem}

\theoremstyle{plain}

\newtheorem{theorem}{Theorem}[section]

\newtheorem{proposition}[theorem]{Proposition}
\newtheorem*{definition*}{Definition}
\numberwithin{equation}{section}

\theoremstyle{definition}

\newtheorem*{remark*}{Remark}

\newcommand*{\variabledot}{\makebox[1ex]{\textbf{$\cdot$}}}

\newcommand{\R}{\mathbb{R}}
\newcommand{\N}{\mathbb{N}}
\newcommand{\C}{\mathbb{C}}
\newcommand{\Z}{\mathbb{Z}}
\renewcommand{\S}{\mathbb{S}}

\newcommand{\indicator}[1]{\mathbf{1}_{#1}}

\newcommand{\intervaloo}[2]{(#1, #2)}
\newcommand{\intervalco}[2]{[#1, #2)}
\newcommand{\intervaloc}[2]{(#1, #2]}
\newcommand{\intervalcc}[2]{[#1, #2]}

\newcommand{\rad}{\textup{rad.}}

\newcommand{\Lrad}{L^2_{\rad}}
\newcommand{\Lperp}{L^2_{\perp}}

\newcommand{\sumn}{\sum_{n=1}^{\dim \HPk{k}(\S^{d-1})}}
\newcommand{\sumjk}{\sum_{\substack{(j, k) \in \N^2 \\ 2j + k = N} } }
\newcommand{\sumjkn}{\sum_{(j, k) \in \N^2} \!\!\!\! \sumn }

\newcommand{\abs}[1]{\lvert#1\rvert} 
\newcommand{\Abs}[1]{\mleft \lvert #1 \mright \rvert}
\newcommand{\norm}[1]{\lVert#1\rVert} 
 
\newcommand{\set}[2]{\{ \, #1 : #2 \, \} } 
\newcommand{\Set}[2]{ \mleft \{ \, #1 : #2 \, \mright\} } 
\newcommand{\conjugate}{\overline}

\newcommand{\Hamiltonian}{H}
\newcommand{\Hamiltonianfree}{\Hamiltonian_0}
\newcommand{\Laplacian}{\Delta}
\newcommand{\LaplacianR}{\Laplacian_{\R^d}}
\newcommand{\LaplacianS}{\Laplacian_{\S^{d-1}}}
\newcommand{\LambdaS}{\Lambda_{\S^{d-1}}}

\newcommand{\Tjnu}[2]{T_{#1}^{(#2)}}
\newcommand{\Tjnuf}[3]{T_{#1}^{(#2)}#3}

\DeclareMathOperator{\supp}{supp}

\newcommand{\const}{\bm{c}}
\newcommand{\Const}{\bm{C}}
\newcommand{\hoconst}{\const}
\newcommand{\hoConst}{\Const}
\newcommand{\hoconstwpd}[3]{\hoconst^{(#3)}(#1, #2)}
\newcommand{\hoConstwpd}[3]{\hoConst^{(#3)}(#1, #2)}
\newcommand{\hoconstwpdrad}[3]{\hoconst_{\rad}^{(#3)}(#1, #2)}
\newcommand{\hoConstwpdrad}[3]{\hoConst_{\rad}^{(#3)}(#1, #2)}
\newcommand{\hoconstwpdperp}[3]{\hoconst_{\perp}^{(#3)}(#1, #2)}
\newcommand{\hoConstwpdperp}[3]{\hoConst_{\perp}^{(#3)}(#1, #2)}

\newcommand{\BesselI}[1]{I_{#1}}

\newcommand{\hypergeometricpFq}[5]{ \,{}_{#1} F_{#2}\mleft( \begin{matrix}
		#3 \\
		#4
	\end{matrix} ; #5 \mright)  }
\newcommand{\hypergeometricU}[3]{U\mleft( \begin{matrix}
		#1 \\
		#2
	\end{matrix} ; #3 \mright)}
\newcommand{\hypergeometricpFqNoArgument}[2]{ \,{}_{#1} F_{#2} }
\newcommand{\hypergeometricUNoArgument}{U}

\newcommand{\WhittakerM}[2]{M_{#1, #2}}
\newcommand{\WhittakerW}[2]{W_{#1, #2}}

\DeclareMathOperator{\csch}{csch}

\newcommand{\hk}[1]{h_{#1}}
\newcommand{\Hk}[1]{H_{#1}}
\newcommand{\LaguerreL}[2]{L_{#1}^{(#2)}}
\newcommand{\LaguerrecalLjnu}[2]{\mathcal{L}_{#1}^{(#2)}}
\newcommand{\Quotient}[2]{Q_{#1}^{(#2)}}
\newcommand{\Ykn}[2]{Y_{#1}^{#2}}
\newcommand{\HPk}[1]{ \mathcal{H}_{#1} }

\newcommand{\gaussian}{\varphi}
\newcommand{\gaussiana}[1]{\gaussian_{#1}}
\newcommand{\weightphia}[1]{\phi_{#1}}

\newcommand{\Productmna}[3]{P(#1,#2;#3)}
\newcommand{\Fnas}[3]{F_{#1, #2}(#3)}
\newcommand{\fjkn}[4]{\widehat{#1}(#2, #3, #4)}

\usepackage{hyperref}
\hypersetup{
	setpagesize=false,
	bookmarksnumbered=true,
	bookmarksopen=true,
	colorlinks=true,
	allcolors=blue,
}

\usepackage[numbers]{natbib}

\title[Optimal smoothing estimates for quantum harmonic oscillators]{Optimal constants of smoothing estimates for quantum harmonic oscillators}
\date{}

\author{Soichiro Suzuki}
\address[Soichiro Suzuki]{Department of Mathematics, Chuo University, 1-13-27 Kasuga, Bunkyo-ku, Tokyo, 112-8551, Japan}
\email{\href{mailto:soichiro.suzuki.m18020a@gmail.com}{soichiro.suzuki.m18020a@gmail.com}}



\begin{document}
		\begin{abstract}
We study optimal constants and extremizers of smoothing estimates for quantum harmonic oscillators.
In particular, we establish harmonic oscillator analogues of free particle results due to \citeauthor{Sim1992} (\citeyear{Sim1992}), \citeauthor{BS2014} (\citeyear{BS2014}), and \citeauthor{BSS2015} (\citeyear{BSS2015}).
	\end{abstract}
	\maketitle
\section{Introduction}
We consider the Schr\"{o}dinger equation for the $d$-dimensional quantum harmonic oscillator
\begin{equation} \label{eq:Schrodinger}
	\begin{cases}
		i \partial_t u(x, t) = \Hamiltonian u(x, t) , & (x, t) \in \R^d \times \R , \\
		u(x, 0) = u_0(x) , & x \in \R^d , \, u_0 \in L^2(\R^d), 
	\end{cases}
\end{equation}
where $\Hamiltonian = \Hamiltonianfree + \abs{x}^2 = - \Laplacian + \abs{x}^2$. 
In 2011, \citet{Che2011} and \citet{BR2011} established the following inequalities, which are natural analogues of classical results for free particles (see \eqref{eq:type B free} and \eqref{eq:type C free}). 
\begin{theorem}[{\cite[Theorem 1]{Che2011}}] \label{thm:Chen}
	Let $d \geq 3$. Then there exists $C > 0$ such that
	\begin{equation} \label{eq:type B HO}
		\int_{(x, t) \in \R^d \times [0, 2 \pi]} \abs{x}^{-2} \abs{  e^{- i t \Hamiltonian} u_0(x) }^2 \, dx \, dt \leq C \norm{ u_0 }_{L^2(\R^d)}^2 
	\end{equation}
	holds for every $u_0 \in L^2(\R^d)$.
\end{theorem}
\begin{theorem}[{\cite[Theorem 3.3]{BR2011}}] \label{thm:BR}
	Let $d \geq 2$ and $p > 1$. Then there exists $C > 0$ such that
	\begin{equation} \label{eq:type C HO}
		\int_{(x, t) \in \R^d \times [0, 2 \pi]} (1 + \abs{x}^2)^{-p/2} \abs{ \Hamiltonian^{1/4} e^{- i t \Hamiltonian} u_0(x) }^2 \, dx \, dt \leq C \norm{ u_0 }_{L^2(\R^d)}^2 
	\end{equation}
	holds for every $u_0 \in L^2(\R^d)$.
\end{theorem}
Our aim in this paper is to determine optimal constants and extremizers of these inequalities.
Here, we say that $u_0 \in L^2(\R^d) \setminus \{0\}$ is an extremizer when it attains the optimal constant.
First, we give our result for the inequality \eqref{eq:type B HO}.
\begin{theorem} \label{main thm:type B HO}
	Let $d \geq 3$. Then the optimal constant of the inequality \eqref{eq:type B HO} is $4\pi / (d-2)$, and $u_0 \in L^2(\R^d) \setminus \{0\}$ is an extremizer if and only if it is radially symmetric.
\end{theorem}
Theorem \ref{main thm:type B HO} is inspired by a similar result for free particles. 
It is classical and well known that 
\begin{equation} \label{eq:type B free}
	\int_{(x, t) \in \R^d \times \R} \abs{x}^{-2} \abs{ e^{- i t \Hamiltonianfree} u_0(x) }^2 \, dx \, dt \leq C \norm{ u_0 }_{L^2(\R^d)}^2 
\end{equation}
holds for every $u_0 \in L^2(\R^d)$ whenever $d \geq 3$, as proved by \citet[Theorem 1]{KY1989}. 
\citet[(3)]{Sim1992} showed that its optimal constant is $\pi / (d-2)$, and \citet[Theorem 1.5]{BS2017} pointed out that $u_0 \in L^2(\R^d) \setminus \{0\}$ is an extremizer if and only if it is radially symmetric.
It is also known that this inequality can be improved by using the spherical Laplacian $\LaplacianS$.
\citet[Theorem 1.2]{Hos1997} showed that 
\begin{equation} \label{eq:type B with spherical Laplacian free}
	\int_{(x, t) \in \R^d \times \R} \abs{x}^{-2} \abs{ (-\LaplacianS + 1)^{1/4} e^{- i t \Hamiltonianfree} u_0(x) }^2 \, dx \, dt \leq C \norm{ u_0 }_{L^2(\R^d)}^2 
\end{equation}
holds, and \citet[Theorem 1.10]{FW2011} established the reversed inequality
\begin{equation} \label{eq:type B with spherical Laplacian reversed free}
	\int_{(x, t) \in \R^d \times \R} \abs{x}^{-2} \abs{ (-\LaplacianS + 1)^{1/4} e^{- i t \Hamiltonianfree} u_0(x) }^2 \, dx \, dt \geq c \norm{ u_0 }_{L^2(\R^d)}^2 .
\end{equation}
Furthermore, \citet{BS2014} studied the optimal constants of these inequalities and found that the identity
\begin{equation} \label{eq:type B identity free d=4}
	\int_{(x, t) \in \R^4 \times \R} \abs{x}^{-2} \abs{ (-\Delta_{\S^{3}} + 1)^{1/4} e^{- i t \Hamiltonianfree} u_0(x) }^2 \, dx \, dt = \frac{\pi}{2} \norm{ u_0 }_{L^2(\R^4)}^2
\end{equation}
holds when $d=4$ (\cite[Theorem 1.2]{BS2014}). 
More generally, following the proof of \cite[Theorem 1.2]{BS2014}, actually it is easy to see that
\begin{equation} \label{eq:type B identity free}
	\int_{(x, t) \in \R^d \times \R} \abs{x}^{-2} \abs{ (-\LaplacianS + (d/2 - 1)^2)^{1/4} e^{- i t \Hamiltonianfree} u_0(x) }^2 \, dx \, dt = \frac{\pi}{2} \norm{ u_0 }_{L^2(\R^d)}^2 
\end{equation}
holds whenever $d \geq 3$.
Inspired by this, we show the following.
\begin{theorem} \label{main thm:type B identity HO}
	Let $d \geq 3$. Then we have
		\begin{equation} \label{eq:type B identity HO}
			\int_{(x, t) \in \R^d \times \intervalcc{0}{2\pi}} \abs{x}^{-2} \abs{ (-\LaplacianS + (d/2 - 1)^2)^{1/4} e^{- i t \Hamiltonian} u_0(x) }^2 \, dx \, dt = 2\pi \norm{u_0}_{L^2(\R^d)}^2
		\end{equation}
		for every $u_0 \in L^2(\R^d)$.
\end{theorem} 
\begin{remark*}
\begin{itemize}[labelindent=\parindent, leftmargin=*]
		\item[]
		\item Theorems \ref{main thm:type B HO} and \ref{main thm:type B identity HO} are actually equivalent to the corresponding results for free particles. See Section \ref{section:lens transform}.
		\item When $d = 2$, the inequalities \eqref{eq:type B HO}, \eqref{eq:type B free} and identities \eqref{eq:type B identity free}, \eqref{eq:type B identity HO} fail for $u_0 \in L^2(\R^2)$ in general. On the other hand, they are still valid if we restrict initial data to a certain subspace of $L^2(\R^2)$. See Section \ref{section:radial initial data}.
	\end{itemize}
\end{remark*}

Now we introduce completely monotone functions and Stieltjes functions in order to state our results on the inequality \eqref{eq:type C HO}.
	\begin{definition*}[completely monotone function]
	A function $f \colon \intervaloo{0}{\infty} \to \R$ is said to be a completely monotone function if one of the following equivalent conditions is satisfied:
	\begin{enumerate}[label=\textup{(CM\arabic*)}]
		\item \label{item:completely monotone 1}
		It is infinitely differentiable and
		satisfies
		\begin{equation}
			(-1)^n f^{(n)}(r) \geq 0
		\end{equation}
		for every $n \in \N \coloneqq \{ 0, 1, 2, \ldots \}$ and $r \in \intervaloo{0}{\infty}$, where $f^{(n)}$ denotes the $n$-th derivative of $f$.
		\item \label{item:completely monotone 2}
		There exists a Borel measure $\lambda$ on $\intervalco{0}{\infty}$ such that
		\begin{equation}
			f(r) = \int_{\rho \in \intervalco{0}{\infty}} \exp(-r \rho) \, d\lambda(\rho)
		\end{equation}
		holds for every $r \in \intervaloo{0}{\infty}$.
	\end{enumerate} 
\end{definition*}
\begin{definition*}[Stieltjes function]
	A function $f \colon \intervaloo{0}{\infty} \to \R$ is said to be a Stieltjes function if one of the following equivalent conditions is satisfied:
	\begin{enumerate}[label=\textup{(S\arabic*)}]
		\item \label{item:Stieltjes 1}
		There exists a Borel measure $\lambda$ on $\intervalco{0}{\infty}$ such that
		\begin{equation}
			f(r) = \int_{a \in \intervalco{0}{\infty}} \frac{1}{a+r} \, d\lambda(a)
		\end{equation}
		holds for every $r \in \intervaloo{0}{\infty}$.
		\item \label{item:Stieltjes 2}
		There exists a completely monotone function $\varphi$ such that
		\begin{equation}
			f(r) = \int_{\rho \in \intervaloo{0}{\infty}} \exp(-r \rho) \varphi(\rho) \, d\rho
		\end{equation}
		holds for every $r \in \intervaloo{0}{\infty}$.
	\end{enumerate}
\end{definition*}
The equivalence $\ref{item:completely monotone 1} \iff \ref{item:completely monotone 2}$ is known as the \citeauthor{Ber1929}--\citeauthor{Wid1931} theorem (see \cite[\textsection 14]{Ber1929} and \cite[Theorem 8]{Wid1931}), and $\ref{item:Stieltjes 1} \iff \ref{item:Stieltjes 2}$ is an immediate consequence of
\begin{equation}
\frac{1}{a+r} = \int_{\rho \in \intervaloo{0}{\infty}} \exp(- (a+r) \rho)  \, d\rho .
\end{equation}
Observe that Stieltjes functions are completely monotone but not conversely.
To illustrate, the function
\begin{equation}
\intervaloo{0}{\infty} \ni r \longmapsto (1 + r)^{-p/2}
\end{equation}
is a Stieltjes function for $p \in \intervaloc{0}{2}$ but not for $p \in \intervaloo{2}{\infty}$, while it is completely monotone for every $p \in \intervaloo{0}{\infty}$.
This follows from the integral representation
\begin{equation}
(1 + r)^{-p/2} = \frac{1}{\Gamma(p/2)} \int_{\rho \in \intervaloo{0}{\infty}} \exp(-r \rho) \rho^{p/2-1} \exp(-\rho) \, d\rho
\end{equation}
and the fact that the function
\begin{equation}
\intervaloo{0}{\infty} \ni \rho \longmapsto \rho^{p/2-1} \exp(-\rho)
\end{equation}
is completely monotone for $p \in \intervaloc{0}{2}$ but not for $p \in \intervaloo{2}{\infty}$.
For further results and related topics on completely monotone and Stieltjes functions, see \citet{Wid1942}, \citet*{SSV2012}, and the references therein, for example.

Now we give our results on the inequality \eqref{eq:type C HO}.
\begin{theorem} \label{main thm:type C HO general}
	Suppose that $w \in L^1(\intervaloo{0}{\infty}) \setminus \{0\}$ is such that
	\begin{equation}
	r \longmapsto w(r^{1/2})
	\end{equation}
	is a Stieltjes function.
	Then, for every $d \geq 3$, the optimal constant of the inequality
	\begin{equation} \label{eq:type C HO general}
		\int_{(x, t) \in \R^d \times \intervalcc{0}{2\pi}} w(\abs{x}) \abs{ \Hamiltonian^{1/4} e^{- i t \Hamiltonian} u_0(x) }^2 \, dx \, dt \leq C \norm{ u_0 }_{L^2(\R^d)}^2 
	\end{equation}
	is $4 \norm{w}_{L^1(\intervaloo{0}{\infty})}$, and there are no extremizers.
	In particular, the optimal constant of the inequality \eqref{eq:type C HO} is
	\begin{equation}
		4 \int_{r \in \intervaloo{0}{\infty}} (1 + r^2)^{-p/2} \, dr = \frac{2 \pi^{1/2} \Gamma( (p-1)/2 ) }{ \Gamma(p/2) } 
	\end{equation}
	when $d \geq 3$ and $p \in \intervaloc{1}{2}$.
\end{theorem}
For free particles, it is well known that
\begin{equation} \label{eq:type C free}
	\int_{(x, t) \in \R^d \times \R} (1 + \abs{x}^2)^{-p/2} \abs{ \Hamiltonianfree^{1/4} e^{- i t \Hamiltonianfree} u_0(x) }^2 \, dx \, dt \leq C \norm{ u_0 }_{L^2(\R^d)}^2 
\end{equation}
holds whenever $p>1$.
This inequality was originally proved by \citet[Theorem 2]{KY1989} for $p=2$ and extended for $p>1$ by \citet[Theorem 1]{BK1992}.
\citet[(2)]{Sim1992} showed that its optimal constant is $\pi / 2$ when $d \geq 3$ and $p=2$.
\citet*{BSS2015} generalized \citeauthor{Sim1992}'s result as follows.
\begin{theorem}[{\cite[Theorem 1.4]{BSS2015}}] \label{thm:type C free general}
		Let $d \geq 3$. Suppose that $w \in L^1(\intervaloo{0}{\infty})$ is a non-negative function such that the $d$-dimensional Fourier transform of $w(\abs{\variabledot})$ is positive everywhere. 
	Then the optimal constant of the inequality
\begin{equation} \label{eq:type C free generalized}
	\int_{(x, t) \in \R^d \times \R} w(\abs{x}) \abs{ \Hamiltonianfree^{1/4} e^{- i t \Hamiltonianfree} u_0(x) }^2 \, dx \, dt \leq C \norm{ u_0 }_{L^2(\R^d)}^2 
\end{equation}
	is $\norm{w}_{L^1(\intervaloo{0}{\infty})}$, and there are no extremizers.
	In particular, the optimal constant of the inequality \eqref{eq:type C free} is
	\begin{equation}
		\int_{r \in \intervaloo{0}{\infty}} (1 + r^2)^{-p/2} \, dr = \frac{\pi^{1/2} \Gamma( (p-1)/2 ) }{ 2 \Gamma(p/2) } 
	\end{equation}
	when $d \geq 3$ and $p \in \intervaloo{1}{\infty}$.
\end{theorem}
Note that if $w \in L^\infty(\intervaloo{0}{\infty}) \setminus \{0\}$ is such that $r \mapsto w(r^{1/2})$ is completely monotone, then $w(\abs{\variabledot})$ is the $d$-dimensional Fourier transform of some radially symmetric Borel measure on $\R^d$ for every $d \geq 1$.
This follows from the definition of complete monotonicity \ref{item:completely monotone 2} and the fact that the Fourier transform of the Gaussian is also the Gaussian. 
Furthermore, the converse is also true (the Schoenberg theorem \cite[Theorem 2]{Sch1938}). 
In this sense, our Theorem \ref{main thm:type C HO general} is an analogue of Theorem \ref{thm:type C free general}.
We remark that both theorems fail when $d=2$; see Section \ref{section:fails in 2D} for Theorem \ref{main thm:type C HO general} and \citet[Theorem 6.2]{Suz2025} for Theorem \ref{thm:type C free general}, respectively.
Moreover, Theorem \ref{main thm:type C HO general} also fails for $w \in L^1(\intervaloo{0}{\infty}) \setminus \{0\}$ such that
\begin{equation}
	r \longmapsto w(r^{1/2})
\end{equation}
is completely monotone in general; see Section \ref{section:fails for cm}.
\section{Preliminaries} \label{section:preliminaries}
In this section, we introduce some well-known results on eigenvalues and eigenfunctions of the spherical Laplacian $-\LaplacianS$ and the Hamiltonian of the quantum harmonic oscillator $\Hamiltonian = - \LaplacianR + \abs{x}^2$.
These facts can be found in standard references; see, for example, \citet{Mul1966, Tha1993, Sze1975}. 
Throughout this section, we assume that $d \geq 2$.
\subsection{Spherical Laplacian and harmonic polynomials} \label{section:spherical Laplacian}
It is well known that the spectrum of the spherical Laplacian $\LaplacianS$ is
\begin{equation} \label{eq:spectrum spherical Laplacian}
\sigma(-\LaplacianS) = \sigma_{\textup{p}}(-\LaplacianS) = \set{ k (k + d - 2) }{ k \in \N } .
\end{equation}
Furthermore, the eigenspace associated with the eigenvalue $k (k + d - 2)$ coincides with the space of spherical harmonic polynomials of degree $k$ in $d$ variables, which will be denoted by $\HPk{k}(\S^{d-1})$.
We note that \eqref{eq:spectrum spherical Laplacian} is equivalent to 
\begin{equation}
\sigma(-\LaplacianS + (d/2-1)^2) = \sigma_{\textup{p}}(-\LaplacianS + (d/2-1)^2) = \set{ (k + d/2 - 1)^2 }{ k \in \N } .
\end{equation}
Hereinafter, we mainly use
\begin{equation}
-\LambdaS \coloneqq -\LaplacianS + (d/2-1)^2
\end{equation}
rather than $-\LaplacianS$ for simplicity. 
\subsection{Quantum harmonic oscillator and associated Laguerre polynomials} \label{section:Laguerre}
It is well known that the spectrum of $\Hamiltonian = - \LaplacianR + \abs{x}^2$ is
\begin{equation}
	\sigma(\Hamiltonian) = \sigma_{\textup{p}}(\Hamiltonian) = \set{ 2N+d }{ N \in \N } .
\end{equation}
\citet{Che2011} and \citet{BR2011} proved Theorems \ref{thm:Chen} and \ref{thm:BR} by using the following orthonormal basis of the associated eigenspace constructed from the Hermite polynomials.
 \begin{proposition} \label{prop:Hermite decomposition}
 	Let $\Hk{k} \colon \R \to \R$ be the Hermite polynomial of degree $k$.
 	Then $\{ \Hk{k} \}_{k \in \N}$ is a complete orthogonal system in $L^2(\R, e^{-x^2} \, dx )$ satisfying
 	\begin{equation}
 	\int_{x \in \R} \abs{\Hk{k}(x)}^2 e^{-x^2} \, dx = \pi^{1/2} 2^k \Gamma(k+1) .
 	\end{equation}
 	Now, for each multi-index $k = (k_j)_{1 \leq j \leq d} \in \N^d$, we define $\hk{k} \colon \R^d \to \R$ by 
 	\begin{equation}
 		\hk{k}(x) \coloneqq \pi^{-d/4} e^{-\abs{x}^2/2} \prod_{j=1}^{d} (  2^{k_j} \Gamma(k_j + 1) )^{-1/2} \Hk{k_j}(x_j) , \quad x = (x_j)_{1 \leq j \leq d} \in \R^d .
 		\label{eq:Hermite functions}
 	\end{equation}
 	Then
 	\begin{equation}
 		\Set{ \hk{k} }{ k \in \N^d \text{ such that } \sum_{j=1}^{d} k_j = N }
 	\end{equation}
 	is an orthonormal basis of the eigenspace of $\Hamiltonian$ associated with the eigenvalue $2N + d$.
 \end{proposition}
However, this basis is not well suited for finding the optimal constants of the inequalities \eqref{eq:type B HO} and \eqref{eq:type C HO}.
The main difficulty here is that $\{ h_k \}_k$ is not orthogonal in the weighted space $L^2(\R^d, w(\abs{x})\, dx)$ generally, even though it is on $L^2(\R^d)$.
Moreover, it is also inconvenient to establish the identity \eqref{eq:type B identity HO}, since we cannot work with both $\Hamiltonian$ and $\LambdaS$ simultaneously.
In order to avoid these problems, we shall use the following basis constructed from the generalized Laguerre polynomials.
\begin{proposition} \label{prop:Laguerre decomposition}
Let $\LaguerreL{j}{\nu} \colon \intervaloo{0}{\infty} \to \R$ be the generalized Laguerre polynomials of degree $j \in \N$ and order $\nu \in \intervaloo{-1}{\infty}$, which can be defined in various ways, for example, by the hypergeometric representation
\begin{equation} \label{eq:Laguerre polynomials}
	\LaguerreL{j}{\nu}(r) 
	\coloneqq \frac{\Gamma(j+\nu+1)}{\Gamma(\nu+1) \Gamma(j+1)} \hypergeometricpFq{1}{1}{-j}{\nu+1}{r} 
	= \frac{\Gamma(j+\nu+1)}{\Gamma(\nu+1) \Gamma(j+1)} e^{r} \hypergeometricpFq{1}{1}{j+\nu+1}{\nu+1}{-r} ,
\end{equation}
where $\hypergeometricpFqNoArgument{1}{1}$ denotes the confluent hypergeometric function of the first kind.
Then $\{ \LaguerreL{j}{\nu} \}_{j \in \N}$ is a complete orthogonal system in $L^2(\intervaloo{0}{\infty}, e^{-r} r^{\nu} \, dr )$ satisfying
\begin{equation}
 	\int_{r \in \intervaloo{0}{\infty}} \abs{\LaguerreL{j}{\nu}(r)}^2 e^{-r} r^{\nu} \, dr = \frac{\Gamma(j+\nu+1)}{\Gamma(j+1)} 
\end{equation}
for each $\nu \in \intervaloo{-1}{\infty}$.
 	Now we define $\LaguerrecalLjnu{j}{\nu} \colon \intervaloo{0}{\infty} \to \R$ by 
\begin{equation}
\LaguerrecalLjnu{j}{\nu}(r) \coloneqq \mleft( \frac{ 2 \Gamma(j+1) }{ \Gamma(j + \nu + 1) } \mright)^{1/2} e^{-r^2 / 2} r^{\nu + 1/2} \LaguerreL{j}{\nu}(r^2) ,
	\label{eq:Laguerre functions} 
\end{equation}
and let $\{ \Ykn{k}{n} \}_n$ be an orthonormal basis of $\HPk{k}(\S^{d-1})$ (the space of spherical harmonic polynomials of degree $k$ in $d$ variables; see Section \ref{section:spherical Laplacian}).
Then
\begin{equation}
\Set{ \abs{x}^{-(d-1)/2} \LaguerrecalLjnu{j}{k + d/2 - 1}(\abs{x}) \Ykn{k}{n}(x/\abs{x}) }{ (j, k, n) \in \N^3 \text{ such that } 2j + k = N, \, 1 \leq n \leq \dim \HPk{k}(\S^{d-1}) }
\end{equation}
 	is an orthonormal basis of the eigenspace of $\Hamiltonian$ associated with the eigenvalue $2N + d$.
\end{proposition}
\subsection{Simultaneous spectral decomposition}
Combining facts in Sections \ref{section:spherical Laplacian} and \ref{section:Laguerre}, we obtain the following simultaneous spectral decomposition.
\begin{proposition} \label{prop:decomposition}
Let $f \in L^2(\R^d)$. For each $(j, k, n) \in \N^3$ satisfying $1 \leq n \leq \dim \HPk{k}(\S^{d-1})$, we write
\begin{equation}
\fjkn{f}{j}{k}{n} \coloneqq \int_{r \in \intervaloo{0}{\infty}} \mleft( \int_{\theta \in \S^{d-1}} f(r \theta) \conjugate{\Ykn{k}{n}(\theta)} \, d\sigma_{\S^{d-1}}(\theta) \mright) r^{(d-1)/2} \LaguerrecalLjnu{j}{k + d/2 - 1}(r) \, dr.
\end{equation}
Then we have
\begin{gather}
f(r \theta) = r^{-(d-1)/2} \sumjkn \fjkn{f}{j}{k}{n} \LaguerrecalLjnu{j}{k + d/2 - 1}(r) \Ykn{k}{n}(\theta) , \quad (r, \theta) \in \intervaloo{0}{\infty} \times \S^{d-1}, \\
\norm{f}_{L^2(\R^d)}^2 = \sumjkn \abs{\fjkn{f}{j}{k}{n}}^2 \eqcolon \norm{\widehat{f}}_{\ell^2}^2 .
\end{gather}
\end{proposition}
This decomposition allows us to define a (possibly unbounded) operator $m(\Hamiltonian, -\LambdaS)$ on $L^2(\R^d)$ for each $m \colon\sigma_{\textup{p}}(\Hamiltonian) \times \sigma_{\textup{p}}(-\LambdaS) \to \C$ via
\begin{align}
& \quad m(\Hamiltonian, -\LambdaS) f \\
&\coloneqq r^{-(d-1)/2} \sumjkn m(4j+2k+d, (k+d/2-1)^2) \fjkn{f}{j}{k}{n} \LaguerrecalLjnu{j}{k + d/2 - 1}(r) \Ykn{k}{n}(\theta) .
\end{align}
\section{Formula for optimal constants} \label{section:formula}
In order to prove our Theorems \ref{main thm:type B HO}, \ref{main thm:type B identity HO}, and \ref{main thm:type C HO general}, we first establish the following formula for the optimal constants, which is an analogue of \citet[Theorem 5.4]{BS2014} for free particles.
\begin{theorem} \label{thm:optimal constant}
	Let $d \geq 2$, $w \colon \intervaloo{0}{\infty} \to \intervalco{0}{\infty}$, and $\psi \colon \sigma(\Hamiltonian) \times \sigma(-\LambdaS) \to \intervalco{0}{\infty}$.
We write $\hoConstwpd{w}{\psi}{d}$ and $\hoconstwpd{w}{\psi}{d}$ for the optimal constants of the following inequalities, respectively:
\begin{align}
	\int_{(x, t) \in \R^d \times \intervalcc{0}{2\pi}} w(\abs{x}) \abs{ \psi(\Hamiltonian, - \LambdaS) e^{- i t \Hamiltonian} u_0(x) }^2 \, dx \, dt &\leq C \norm{u_0}_{L^2(\R^d)}^2 , 
	\label{eq:smoothing estimate general} \\
	\int_{(x, t) \in \R^d \times \intervalcc{0}{2\pi}} w(\abs{x}) \abs{ \psi(\Hamiltonian, - \LambdaS) e^{- i t \Hamiltonian} u_0(x) }^2 \, dx \, dt &\geq c \norm{u_0}_{L^2(\R^d)}^2 .
	\label{eq:smoothing estimate general reverse}
\end{align}
For each $\nu \in \intervaloo{-1}{\infty}$ and $j \in \N$, we define a linear functional $\Tjnu{j}{\nu}$ by
\begin{equation}
\Tjnuf{j}{\nu}{f} \coloneqq \int_{r \in \intervaloo{0}{\infty}} f(r) \LaguerrecalLjnu{j}{\nu}(r)^2 \, dr .
\end{equation}
Then we have
\begin{alignat}{2}
	\hoConstwpd{w}{\psi}{d} &={}& 2\pi \sup_{j \in \N} \sup_{k \in \N} {}&\psi(4j+2k+d, (k+d/2-1)^2)^2 \Tjnuf{j}{k+d/2-1}{w} , 
	\label{eq:ho optimal constant C} \\
	\hoconstwpd{w}{\psi}{d} &={}&2\pi \inf_{j \in \N} \inf_{k \in \N} {}&\psi(4j+2k+d, (k+d/2-1)^2)^2 \Tjnuf{j}{k+d/2-1}{w} .
	\label{eq:ho optimal constant c}
\end{alignat}
Furthermore, $u_0 \in L^2(\R^d) \setminus \{0\}$ is an extremizer of \eqref{eq:smoothing estimate general}
if and only if 
$\fjkn{u_0}{j}{k}{n} = 0$
for every $(j, k, n)$ such that
\begin{equation}
2\pi \psi(4j+2k+d, (k+d/2-1)^2)^2 \Tjnuf{j}{k+d/2-1}{w} < \hoConstwpd{w}{\psi}{d} .
\end{equation}
Similarly, 
$u_0 \in L^2(\R^d) \setminus \{0\}$ is an extremizer of \eqref{eq:smoothing estimate general reverse}
if and only if 
$\fjkn{u_0}{j}{k}{n} = 0$
for every $(j, k, n)$ such that
\begin{equation}
2\pi \psi(4j+2k+d, (k+d/2-1)^2)^2 \Tjnuf{j}{k+d/2-1}{w} > \hoconstwpd{w}{\psi}{d} .
\end{equation}
\end{theorem}

\begin{proof}[Proof of Theorem \ref{thm:optimal constant}]
Throughout the proof, we use polar coordinates $x = r \theta$, $(r, \theta) \in \intervaloo{0}{\infty} \times \S^{d-1}$.
Fix $u_0 \in L^2(\R^d)$. By Proposition \ref{prop:decomposition}, we have
\begin{align}
&\quad \psi(\Hamiltonian, -\LambdaS) e^{- i t \Hamiltonian} u_0(x) \\
&= r^{-(d-1)/2} \sum_{N \in \N} \sumjk \!\!\!\!\!\! \sumn \psi(2N+d, (k+d/2-1)^2) e^{- i (2N+d) t} \fjkn{u_0}{j}{k}{n} \LaguerrecalLjnu{j}{k + d/2 - 1}(r) \Ykn{k}{n}(\theta) .
\end{align}
Then we get
\begin{align}
&\quad \int_{t \in \intervalcc{0}{2\pi}} \abs{ \psi(\Hamiltonian, -\LambdaS) e^{- i t \Hamiltonian} u_0(x) }^2 \, dt \\
&= 2\pi r^{-(d-1)} \sum_{N \in \N} \Abs{ \sumjk \!\!\!\!\!\! \sumn \psi(2N+d, (k+d/2-1)^2) \fjkn{u_0}{j}{k}{n} \LaguerrecalLjnu{j}{k+d/2-1}(r) \Ykn{k}{n}(\theta) }^2 ,
\end{align}
since $\{ e^{-i (2N+d) t} / \sqrt{2\pi} \}_{N \in \N}$ is an orthonormal system in $L^2(\intervalcc{0}{2\pi})$.
Furthermore, we also have
\begin{align}
&\quad \int_{\theta \in \S^{d-1}} \Abs{ \sumjk \!\!\!\!\!\! \sumn \psi(2N+d,(k+d/2-1)^2) \fjkn{u_0}{j}{k}{n} \LaguerrecalLjnu{j}{k+d/2-1}(r) \Ykn{k}{n}(\theta) }^2 \, d\sigma_{\S^{d-1}}(\theta) \\
& = \sumjk \!\!\!\!\!\! \sumn \psi(2N+d, (k+d/2-1)^2)^2 \abs{\fjkn{u_0}{j}{k}{n}}^2 \LaguerrecalLjnu{j}{k+d/2-1}(r)^2
\end{align}
for each $N \in \N$, since $\{ \Ykn{k}{n} \}_{k, n}$ is an orthonormal system in $L^2(\S^{d-1})$.
Combining these, we get
\begin{align}
&\quad \int_{(x, t) \in \R^d \times \intervalcc{0}{2\pi}} w(\abs{x}) \abs{ \psi(\Hamiltonian, - \LambdaS) e^{- i t \Hamiltonian} u_0(x) }^2 \, dx \, dt \\
&= 2\pi \sum_{N \in \N} \sumjk \!\!\!\!\!\! \sumn \psi(2N+d, (k+d/2-1)^2)^2 \abs{\fjkn{u_0}{j}{k}{n}}^2 \int_{r \in \intervaloo{0}{\infty}} w(r) \LaguerrecalLjnu{j}{k+d/2-1}(r)^2 \, dr \\
&= 2\pi \sumjkn \psi(4j+2k+d, (k+d/2-1)^2)^2 \Tjnuf{j}{k+d/2 - 1}{w} \abs{\fjkn{u_0}{j}{k}{n}}^2 .
\label{eq:LHS}
\end{align}
Now Theorem \ref{thm:optimal constant} follows from this identity \eqref{eq:LHS} and $\norm{u_0}_{L^2(\R^d)} = \norm{\widehat{u_0}}_{\ell^2}$. 
\end{proof}
\section{Proofs of Theorems \ref{main thm:type B HO} and \ref{main thm:type B identity HO}} \label{section:proof type B}
Theorems \ref{main thm:type B HO} and \ref{main thm:type B identity HO} easily follow from Theorem \ref{thm:optimal constant}.
\begin{proof}[Proof of Theorems \ref{main thm:type B HO} and \ref{main thm:type B identity HO}]
Let $w(r) = r^{-2}$. Then, for each $\nu \in \intervaloo{0}{\infty}$, we have
\begin{align}
	\Tjnuf{j}{\nu}{w}
	&= \frac{2\Gamma(j+1)}{\Gamma(j+\nu+1)} \int_{r \in \intervaloo{0}{\infty}} \exp( -r^2 ) r^{2\nu-1} \LaguerreL{j}{\nu}(r^2)^2 \, dr \\
	&= \frac{\Gamma(j+1)}{\Gamma(j+\nu+1)} \int_{r \in \intervaloo{0}{\infty}} \exp( -r ) r^{\nu-1} \LaguerreL{j}{\nu}(r)^2 \, dr \\
	\underset{\text{\cite[2.19.14.17]{PBM1988}}}&= %\frac{\Gamma(j+1)}{\Gamma(j+\nu+1)} \frac{\Gamma(j+\nu+1) \Gamma(\nu)}{\Gamma(\nu+1) \Gamma(j+1)} \\&= 
	\frac{1}{\nu} .
\end{align}
Therefore, the optimal constant of the inequality \eqref{eq:type B HO} is given by
\begin{equation}
2\pi \sup_{j \in \N} \sup_{k \in \N} \Tjnuf{j}{k+d/2 - 1}{w} = \sup_{j \in \N} \sup_{k \in \N} \frac{4\pi}{2k+d-2} = \frac{4\pi}{d-2} .
\end{equation}
Moreover, $u_0 \in L^2(\R^d) \setminus \{0\}$ is an extremizer if and only if it is radially symmetric, since the supremum above is attained if and only if $k = 0$.
This proves Theorem \ref{main thm:type B HO}.
Similarly, Theorem \ref{main thm:type B identity HO} holds because we have
\begin{equation}
2\pi (k+d/2-1) \Tjnuf{j}{k+d/2 - 1}{w} = 2 \pi 
\end{equation}
for every $(j, k) \in \N^2$.
\end{proof}
\section{Proof of Theorem \ref{main thm:type C HO general}} \label{section:proof type C}
In this section, we prove Theorem \ref{main thm:type C HO general}. We note that our argument here is inspired by \citet{Suz2025_2}, which studies estimates for free particles associated with a weight $w$ such that $r \longmapsto w(r^{1/2})$ is completely monotone.
We first prove the following Proposition \ref{prop:asymptotic behavior}, which gives a lower bound for the optimal constant of the inequality \eqref{eq:type C HO general}.
\begin{proposition} \label{prop:asymptotic behavior}
Let $\nu \in \intervalco{-1/2}{\infty}$ and $f \in L^1(\intervaloo{0}{\infty})$. Then the following hold.
\begin{enumerate}[label=\textup{(\roman*)}]
\item \label{item:asymptotic behavior for compactly supported f} 
If there exists a compact set $K \subset \intervaloo{0}{\infty}$ such that $\supp f \subset K$, then
\begin{equation} \label{eq:asymptotic behavior for compactly supported f} 
	\lim_{j \to \infty} (4j+2\nu+2)^{1/2} \Tjnuf{j}{\nu}{f} = \frac{2}{\pi} \int_{0}^{\infty} f(r) \, dr .
\end{equation}
\item \label{item:asymptotic behavior for non-negative f} 
 If $f$ is non-negative, then 
\begin{equation} \label{eq:asymptotic behavior for non-negative f} 
	\liminf_{j \to \infty} (4j+2\nu+2)^{1/2} \Tjnuf{j}{\nu}{f} \geq \frac{2}{\pi} \int_{0}^{\infty} f(r) \, dr .
\end{equation}
\end{enumerate}
\end{proposition}
In order to prove this, we use the following asymptotic formula for the Laguerre polynomials.
\begin{proposition}[{\cite[Theorem 8.22.1]{Sze1975}}]
	Let $\nu \in \intervaloo{-1}{\infty}$, and let $K \subset \intervaloo{0}{\infty}$ be a compact set. Then
	\begin{equation} \label{eq:asymptotic behavior for Laguerre polynomial}
		\LaguerreL{j}{\nu}(r) = \frac{j^{\nu/2-1/4} e^{r/2}}{\pi^{1/2} r^{\nu/2+1/4}} \cos{ (2(jr)^{1/2} - (\nu/2+1/4)\pi )} + O(j^{\nu/2-3/4})
	\end{equation}
	holds uniformly on $K$ as $j \to \infty$. 
\end{proposition}
\begin{proof}[Proof of Proposition \ref{prop:asymptotic behavior}]
	First, we prove \ref{item:asymptotic behavior for compactly supported f}. 
	Suppose that $f \in L^1(\intervaloo{0}{\infty})$ has a compact support $K$.
Then, combining $\Gamma(j+1)/\Gamma(j+\nu+1) \sim j^{-\nu}$ and the asymptotic formula \eqref{eq:asymptotic behavior for Laguerre polynomial}, we see that
	\begin{equation} \label{eq:asymptotic behavior for Laguerre function}
	(4j+2\nu+2)^{1/2} \LaguerrecalLjnu{j}{\nu}(r)^2 = \frac{2}{\pi} \sin{(4j^{1/2} r -\nu\pi)} + \frac{2}{\pi} + O(j^{-1/2})
\end{equation}
holds uniformly on $K$ as $j \to \infty$. Therefore, we have
\begin{equation}
(4j+2\nu+2)^{1/2} \Tjnuf{j}{\nu}{f} = \frac{2}{\pi} \int_{r \in K} f(r) \sin{(4j^{1/2} r -\nu\pi)} \, dr + \frac{2}{\pi} \int_{r \in K} f(r) \, dr + O(j^{-1/2})
\end{equation}
as $j \to \infty$.
Using the Riemann--Lebesgue lemma, we obtain
\begin{equation}
\lim_{j \to \infty} (4j+2\nu+2)^{1/2} \Tjnuf{j}{\nu}{f} = \frac{2}{\pi} \int_{r \in K} f(r) \, dr ,
\end{equation}
as desired. 

Next, we prove \ref{item:asymptotic behavior for non-negative f}. 
Suppose that $f \in L^1(\intervaloo{0}{\infty})$ is non-negative and
fix $\varepsilon > 0$ arbitrarily. 
Let $K \subset \intervaloo{0}{\infty}$ be a compact set such that
\begin{equation}
\frac{2}{\pi} \int_{r \in K} f(r) \, dr > \frac{2}{\pi} \int_{r \in \intervaloo{0}{\infty}} f(r) \, dr - \varepsilon / 2 ,
\end{equation}
and let $j_0 \in \N$ be such that 
\begin{equation}
(4j+2\nu+2)^{1/2} \int_{r \in K} f(r) \LaguerrecalLjnu{j}{\nu}(r)^2 \, dr > \frac{2}{\pi} \int_{r \in K} f(r) \, dr - \varepsilon / 2
\end{equation}
holds for every $j \in \N_{\geq j_0}$, which exists thanks to \ref{item:asymptotic behavior for compactly supported f}.
Then we have
\begin{align}
(4j+2\nu+2)^{1/2} \int_{r \in \intervaloo{0}{\infty}} f(r) \LaguerrecalLjnu{j}{\nu}(r)^2 \, dr
&\geq (4j+2\nu+2)^{1/2} \int_{r \in K} f(r) \LaguerrecalLjnu{j}{\nu}(r)^2 \, dr \\
&> \frac{2}{\pi} \int_{r \in K} f(r) \, dr - \varepsilon / 2  \\
&> \frac{2}{\pi} \int_{r \in \intervaloo{0}{\infty}} f(r) \, dr - \varepsilon
\end{align}
for every $j \in \N_{\geq j_0}$. This shows \ref{item:asymptotic behavior for non-negative f}.
\end{proof}
Combining Theorem \ref{thm:optimal constant} and Proposition \ref{prop:asymptotic behavior}, we see that the optimal constant of the inequality \eqref{eq:type C HO general} is greater than or equal to $4\norm{w}_{L^1(\intervaloo{0}{\infty})}$ for every $d \geq 2$ and non-negative $w \in L^1(\intervaloo{0}{\infty})$.
In particular, as conjectured by \citet{BR2011}, the inequality \eqref{eq:type C HO} fails when $p \leq 1$.
More precisely, this follows by applying the lower bound to the truncated weights
$(1+r^2)^{-p/2} \indicator{\intervalco{0}{R}}$ and letting $R \to \infty$.

Now we consider upper bounds. In order to establish Theorem \ref{main thm:type C HO general}, it suffices to show the following thanks to Theorem \ref{thm:optimal constant}.
\begin{theorem} \label{thm:monotonicity for Stieltjes weights}
Let $f \in L^1(\intervaloo{0}{\infty}) \setminus \{0\}$ be such that 
\begin{equation}
\intervaloo{0}{\infty} \ni r \longmapsto f(r^{1/2})
\end{equation}
is a Stieltjes function.
Then, for each $\nu \in \intervalco{1/2}{\infty}$,
\begin{equation}
\N \ni j \longmapsto (4j+2\nu+2)^{1/2} \Tjnuf{j}{\nu}{f} 
\end{equation}
is strictly increasing and converges to $\frac{2}{\pi} \norm{f}_{L^1(\intervaloo{0}{\infty})}$ as $j \to \infty$.
\end{theorem}
Hereinafter, we write
\begin{equation}
	\weightphia{a}(r) \coloneqq (a + r^2)^{-1} 
\end{equation}
for each $a \in \intervaloo{0}{\infty}$. 
First, we show that $\Tjnuf{j}{\nu}{\weightphia{a}}$ can be represented as a product of the Whittaker functions.
We recall that the Whittaker function of the first kind $\WhittakerM{\mu}{\nu}$ and of the second kind $\WhittakerW{\mu}{\nu}$ are given by
\begin{align}
	\WhittakerM{\mu}{\nu}(r) &\coloneqq e^{-r/2} r^{\nu+1/2} \hypergeometricpFq{1}{1}{\nu-\mu+1/2}{2\nu+1}{r} , 
	\label{eq:WhittakerM}\\
	\WhittakerW{\mu}{\nu}(r) &\coloneqq e^{-r/2} r^{\nu+1/2} \hypergeometricU{\nu-\mu+1/2}{2\nu+1}{r} ,  
	\label{eq:WhittakerW}
\end{align}
respectively, where $\hypergeometricpFqNoArgument{1}{1}$ and $\hypergeometricUNoArgument$ are the confluent hypergeometric functions of the first and second kinds (see \cite[\href{https://dlmf.nist.gov/13\#PT3}{13.14--13.26}]{DLMF} and \cite[9.22--9.23]{GR2014}, for example).
In what follows, we write
	\begin{equation} \label{eq:products of Whittaker functions}
	\Productmna{\mu}{\nu}{a} \coloneqq \frac{ \Gamma(\mu/2 + \nu/2 + 1/2)}{a \Gamma(\nu+1)} \WhittakerM{-\mu/2}{\nu/2}(a) \WhittakerW{-\mu/2}{\nu/2}(a) .
\end{equation}
Then we have the following.
\begin{proposition} \label{prop:explicit formula}
Let $j \in \N$, $\nu \in \intervaloo{-1}{\infty}$, and $a \in \intervaloo{0}{\infty}$.
Then we have
\begin{equation} \label{eq:explicit formula}
	\Tjnuf{j}{\nu}{\weightphia{a}} = \Productmna{2j+\nu+1}{\nu}{a} .
\end{equation}
\end{proposition}
\begin{proof}[Proof of Proposition \ref{prop:explicit formula}]
Let $\Quotient{j-1}{\nu}(a; r)$ be the quotient polynomial obtained by dividing $\LaguerreL{j}{\nu}(r)$ by $a + r$, that is, 
\begin{equation} \label{eq:quotient}
\Quotient{j-1}{\nu}(a;r) = \frac{\LaguerreL{j}{\nu}(r) - \LaguerreL{j}{\nu}(-a)}{a+r} .
\end{equation}
Then the orthogonality of the Laguerre polynomials implies
\begin{equation} \label{eq:orthogonality of Q and L}
\frac{ \Gamma(j+1) }{\Gamma(j+\nu+1)} \int_{r \in \intervaloo{0}{\infty}} \exp(-r) r^\nu \LaguerreL{j}{\nu}(r) \Quotient{j-1}{\nu}(a; r) \, dr = 0 ,
\end{equation}
since $\Quotient{j-1}{\nu}$ is a polynomial of degree $j-1$. Thus, we have
\begin{align}
	&\quad\Tjnuf{j}{\nu}{\weightphia{a}}\\
	&= \frac{ \Gamma(j+1) }{\Gamma(j+\nu+1)} \int_{r \in \intervaloo{0}{\infty}} (a+r)^{-1} \exp(-r) r^\nu \LaguerreL{j}{\nu}(r)^2 \, dr \\
	\underset{\eqref{eq:quotient}}&=
	\frac{ \Gamma(j+1) }{\Gamma(j+\nu+1)} \int_{r \in \intervaloo{0}{\infty}} \exp(-r) r^\nu \LaguerreL{j}{\nu}(r) (\Quotient{j-1}{\nu}(a;r) + (a+r)^{-1} \LaguerreL{j}{\nu}(-a)) \, dr \\
	\underset{\eqref{eq:orthogonality of Q and L}}&= \frac{ \Gamma(j+1) }{\Gamma(j+\nu+1)} \LaguerreL{j}{\nu}(-a) \int_{r \in \intervaloo{0}{\infty}} (a+r)^{-1} \exp(-r) r^\nu \LaguerreL{j}{\nu}(r) \, dr .
	\label{eq:explicit formula step 1}
\end{align}
Now we use the hypergeometric representation \eqref{eq:Laguerre polynomials}
and the integral formula \cite[7.623.1]{GR2014}
\begin{gather} \label{eq:explicit formula step 2}
\int_{r \in \intervaloo{0}{\infty}} (a+r)^{-1} \exp(-r) r^{\nu} \hypergeometricpFq{1}{1}{-\mu}{\nu+1}{r} \, dr 
= \Gamma(\mu+1) \Gamma(\nu + 1) a^{\nu} \hypergeometricU{\mu + \nu + 1}{\nu + 1}{a} ,
\end{gather}
which is valid for every $\mu, \nu \in \intervaloo{-1}{\infty}$ and $a \in \intervaloo{0}{\infty}$.
Then we get
\begin{align}
	&\quad\Tjnuf{j}{\nu}{\weightphia{a}}\\
	\underset{\eqref{eq:Laguerre polynomials}, \eqref{eq:explicit formula step 1}}&=
	\frac{ \Gamma(j+\nu+1)}{\Gamma(j+1) \Gamma(\nu+1)^2} e^{-a} \hypergeometricpFq{1}{1}{j+\nu+1}{\nu+1}{a} \int_{r \in \intervaloo{0}{\infty}} (a+r)^{-1} \exp(-r) r^\nu \hypergeometricpFq{1}{1}{-j}{\nu+1}{r} \, dr \\
	\underset{\eqref{eq:explicit formula step 2}}&= \frac{ \Gamma(j+\nu+1)}{\Gamma(\nu+1)} e^{-a} a^{\nu} \hypergeometricpFq{1}{1}{j+\nu+1}{\nu+1}{a} \hypergeometricU{j + \nu + 1}{\nu + 1}{a} \\
	\underset{\eqref{eq:WhittakerM}, \eqref{eq:WhittakerW}}&= \frac{\Gamma(j+\nu+1)}{a \Gamma(\nu+1)} \WhittakerM{-j-\nu/2-1/2}{\nu/2}(a) \WhittakerW{-j-\nu/2-1/2}{\nu/2}(a), 
\end{align}
as desired.
\end{proof}
Next, we show some monotonicity results for $\Productmna{\mu}{\nu}{a}$ with respect to $\mu$ and $\nu$.
\begin{proposition} \label{prop:products of Whittaker functions}
	The following hold.
	\begin{enumerate}[label=\textup{(\roman*)}]
		\item \label{item:decreasing with respect to nu}
		For each $\mu \in \intervaloo{-1}{\infty}$ and $a \in \intervaloo{0}{\infty}$, 
		\begin{equation}
	\intervalco{0}{\infty} \ni \nu \longmapsto \Productmna{\mu}{\nu}{a}
\end{equation}	
is strictly decreasing.
		\item \label{item:decreasing with respect to mu}
		For each $\nu \in \intervalco{0}{\infty}$ and $a \in \intervaloo{0}{\infty}$, 
\begin{equation}
	\intervaloo{-\nu-1}{\infty} \ni \mu \longmapsto \Productmna{\mu}{\nu}{a}
\end{equation}	
is completely monotone.
		\item  \label{item:increasing with respect to mu}
		For each $\nu \in \intervalco{1/2}{\infty}$ and $a \in \intervaloo{0}{\infty}$,
		\begin{equation}
	\intervalco{0}{\infty} \ni \mu \longmapsto (2 \mu a)^{1/2} \Productmna{\mu}{\nu}{a}
		\end{equation}
		is strictly increasing and converges to $1$ as $\mu \to \infty$.
	\end{enumerate}
\end{proposition}
\begin{proof}[Proof of Proposition \ref{prop:products of Whittaker functions}]
We begin with the following integral representation (\cite[6.669.4]{GR2014}):
\begin{align} \label{eq:products of Whittaker functions integral form}
&\quad \frac{\Gamma(\mu/2 + \nu/2 + 1/2)}{(ab)^{1/2} \Gamma(\nu+1)} \WhittakerW{-\mu/2}{\nu/2}(a) \WhittakerM{-\mu/2}{\nu/2}(b) \\
&= \int_{t \in \intervaloo{0}{\infty}} \exp\mleft(- \frac{a+b}{2} \cosh{t}\mright) \BesselI{\nu}((ab)^{1/2} \sinh{t}) ( \tanh(t/2) )^{\mu} \, dt ,
\end{align}
which is valid whenever\footnote{To be precise, the formula in \citet[6.669.4]{GR2014} is stated for $\mu + \nu + 1 > 0$, $\nu > 0$, $a > b > 0$. 
Nevertheless, we can easily extend it to $\nu = 0$ and $a = b > 0$ by a standard limiting argument.}
\begin{equation}
\mu + \nu + 1 > 0 , \quad \nu \geq 0 , \quad a \geq b > 0 .
\end{equation}
Here, $\BesselI{\nu}$ denotes the modified Bessel function of the first kind of order $\nu$, which is strictly positive on $\intervaloo{0}{\infty}$.
Letting $a = b$ and using the substitution
\begin{equation}
	\tanh(t/2) = \exp(-s) ,
\end{equation}
we get
\begin{equation}
	\Productmna{\mu}{\nu}{a}
	= (2 \pi a)^{-1/2} \int_{s \in \intervaloo{0}{\infty}} s^{-1/2} \Fnas{\nu}{a}{s} \exp(-\mu s)  \, ds ,
\end{equation}
where
\begin{equation}
	\Fnas{\nu}{a}{s} \coloneqq (2 \pi a s)^{1/2} \exp(- a \coth{s}) \BesselI{\nu}(a \csch{s}) \csch{s} .
\end{equation}
The property \ref{item:decreasing with respect to mu} is immediate from this expression.
Moreover, \ref{item:decreasing with respect to nu} follows from the fact that
\begin{equation}
\intervalco{0}{\infty} \ni \nu \longmapsto \BesselI{\nu}(r)
\end{equation}
is strictly decreasing for each $r \in \intervaloo{0}{\infty}$ (see \citet{Coc1967, Reu1968, Jon1968}).
In order to prove \ref{item:increasing with respect to mu}, we write
\begin{align} \label{eq:products of Whittaker functions integral form 2}
	(2 \mu a)^{1/2} \Productmna{\mu}{\nu}{a}
	&= \pi^{-1/2} \mu^{1/2} \int_{s \in \intervaloo{0}{\infty}} s^{-1/2} \Fnas{\nu}{a}{s} \exp(-\mu s)  \, ds \\
	&= \pi^{-1/2} \int_{s \in \intervaloo{0}{\infty}} s^{-1/2} \Fnas{\nu}{a}{s / \mu} \exp(-s)  \, ds .
\end{align}
Since we know that
\begin{equation}
\int_{s \in \intervaloo{0}{\infty}} s^{-1/2} \exp(-s)  \, ds = \Gamma(1/2) = \pi^{1/2},
\end{equation}
it suffices to show that the function
\begin{equation}
\intervaloo{0}{\infty} \ni s \longmapsto \Fnas{\nu}{a}{s}
\end{equation}
is strictly decreasing and satisfies
\begin{equation}
	\lim_{s \downarrow 0} \Fnas{\nu}{a}{s} = 1
\end{equation}
for each $\nu \in \intervalco{1/2}{\infty}$ and $a \in \intervaloo{0}{\infty}$.
To see this, recall that
\begin{equation}
	\BesselI{1/2}(x) = \mleft( \frac{2}{\pi x} \mright)^{1/2} \sinh{x} ,
\end{equation}
so that we have
\begin{equation}
\Fnas{\nu}{a}{s} = 2 (s \csch{s})^{1/2} \exp(- a \coth{s}) \sinh(a \csch{s}) \frac{\BesselI{\nu}(a \csch{s})}{\BesselI{1/2}(a \csch{s})} .
\end{equation}
It is easy to see that the functions
\begin{equation}
\intervaloo{0}{\infty} \ni s \longmapsto s \csch{s} = \frac{2s}{e^{s} - e^{-s}}
\end{equation}
and
\begin{align}
\intervaloo{0}{\infty} \ni s \longmapsto 2 \exp(- a \coth{s}) \sinh(a \csch{s})
	&= \exp( - a (\coth{s} - \csch{s}) ) - \exp( - a (\coth{s} + \csch{s}) ) \\
	&= \exp( - a \tanh{(s/2)}  ) - \exp( - a \coth{(s/2)} )
\end{align}
are strictly decreasing and converge to $1$ as $s \downarrow 0$.
This proves the desired result for the case $\nu = 1/2$. 
When $\nu \in \intervaloo{1/2}{\infty}$, we use the fact that
\begin{equation}
x \longmapsto \frac{\BesselI{\nu_1}(x)}{\BesselI{\nu_2}(x)}
\end{equation}
is strictly increasing on $\intervaloo{0}{\infty}$ and converges to $1$ as $x \to \infty$ whenever $\nu_1 > \nu_2 \geq 0$; see \citet[Theorem 1]{Lor1967} or \citet[Proposition 7.1]{HW1974}.
Thus, the function
\begin{equation}
\intervaloo{0}{\infty} \ni s \longmapsto \frac{\BesselI{\nu}(a \csch{s})}{\BesselI{1/2}(a \csch{s})}
\end{equation}
is strictly decreasing and converges to $1$ as $s \downarrow 0$ when $\nu \in \intervaloo{1/2}{\infty}$.
This completes the proof.
\end{proof}
Now we prove Theorem \ref{thm:monotonicity for Stieltjes weights} using Propositions \ref{prop:explicit formula} and \ref{prop:products of Whittaker functions}.
\begin{proof}[Proof of Theorem \ref{thm:monotonicity for Stieltjes weights}]
Let $f \in L^1(\intervaloo{0}{\infty}) \setminus \{0\}$ be such that 
\begin{equation}
	r \longmapsto f(r^{1/2})
\end{equation}
is a Stieltjes function. 
Then, by the definition of Stieltjes functions \ref{item:Stieltjes 1} and the assumption $f \in L^1(\intervaloo{0}{\infty}) \setminus \{0\}$, 
there exists a non-zero Borel measure $\lambda$ on $\intervaloo{0}{\infty}$ such that
\begin{equation} \label{eq:integral representation of weight}
	f(r) = \int_{a \in \intervaloo{0}{\infty}} \weightphia{a}(r) \, d\lambda(a)
\end{equation}
holds for every $r \in \intervaloo{0}{\infty}$.
Therefore, for each $\nu \in \intervaloo{-1}{\infty}$ and $j \in \N$, we have
\begin{equation} \label{eq:formula for Tjnu}
\Tjnuf{j}{\nu}{f} 
\underset{\eqref{eq:integral representation of weight}}{=} \int_{a \in \intervaloo{0}{\infty}} \Tjnuf{j}{\nu}{\weightphia{a}} \, d\lambda(a) 
\underset{\text{Proposition \ref{prop:explicit formula}}}{=} \int_{a \in \intervaloo{0}{\infty}} \Productmna{2j + \nu + 1}{\nu}{a} \, d\lambda(a)
\end{equation}
and 
\begin{equation} \label{eq:formula for L1 norm}
\norm{f}_{L^1(\intervaloo{0}{\infty})} = \int_{r \in \intervaloo{0}{\infty}} f(r) \, dr \underset{\eqref{eq:integral representation of weight}}{=} \int_{a \in \intervaloo{0}{\infty}} \int_{r \in \intervaloo{0}{\infty}} \weightphia{a}(r) \, dr \, d\lambda(a) = \frac{\pi}{2} \int_{a \in \intervaloo{0}{\infty}} a^{-1/2} \, d\lambda(a) .
\end{equation}
On the other hand, by Proposition \ref{prop:products of Whittaker functions}, 
\begin{equation}
\N \ni j \longmapsto (4j+2\nu+2)^{1/2} \Productmna{2j + \nu + 1}{\nu}{a}
\end{equation}
is strictly increasing and converges to $a^{-1/2}$ as $j \to \infty$ for each $\nu \in \intervalco{1/2}{\infty}$ and $a \in \intervaloo{0}{\infty}$.
Hence, we conclude that 
\begin{equation}
j \longmapsto (4j+2\nu+2)^{1/2} \Tjnuf{j}{\nu}{f} = \int_{a \in \intervaloo{0}{\infty}} (4j+2\nu+2)^{1/2} \Productmna{2j + \nu + 1}{\nu}{a} \, d\lambda(a)
\end{equation}
is strictly increasing and converges to $\frac{2}{\pi} \norm{f}_{L^1(\intervaloo{0}{\infty})}$ as $j \to \infty$ for each $\nu \in \intervalco{1/2}{\infty}$.
\end{proof}
\section{Remarks}
\subsection{A relation between free particles and harmonic oscillators} \label{section:lens transform}
We remark that Theorems \ref{main thm:type B HO} and \ref{main thm:type B identity HO} are equivalent to the corresponding results for the free particle.
Indeed, the left-hand sides of \eqref{eq:type B HO} and \eqref{eq:type B identity HO} are equal to four times those of \eqref{eq:type B free} and \eqref{eq:type B identity free}, respectively. More generally, we have the identity
\begin{equation} \label{eq:relation between free and HO}
	\begin{aligned} 
		&\quad \int_{(x, t) \in \R^d \times [0, 2 \pi]} w(\abs{x}) \abs{ \psi(-\LambdaS) e^{- i t \Hamiltonian} u_0(x) }^2 \, dx \, dt \\
		&= 4 \int_{(\xi, \tau) \in \R^d \times \R} \frac{w(\abs{\xi} / ( 1 + 4 \tau^2 )^{1/2})}{1 + 4\tau^2} \abs{ \psi(-\LambdaS) e^{- i \tau \Hamiltonianfree} u_0(\xi) }^2 \, d\xi \, d\tau 
	\end{aligned}
\end{equation}
for every $w \colon \intervaloo{0}{\infty} \to \intervalco{0}{\infty}$, $\psi \colon \sigma(- \LambdaS) \to \C$, and $u_0 \in L^2(\R^d)$.
This can be shown by using the so-called lens transform (also known as the \citeauthor{Nie1973} transform \cite[Section~4]{Nie1973}). Namely, we have
\begin{equation} \label{eq:lens transform} \noeqref{eq:lens transform}
	( e^{- i t \Hamiltonian} u_0 ) (x) = e^{- i n d \pi / 2} (1 + 4 \tau^2)^{d/4} \exp\mleft( - i \frac{\tau}{1 + 4 \tau^2} \abs{\xi}^2 \mright) ( e^{- i \tau \Hamiltonianfree} u_0 )(\xi),
\end{equation}
where
\begin{equation}
	n \in \Z, \quad 
	x \in \R^d , \quad t \in I_n \coloneqq \intervaloo{\tfrac{n \pi}{2} - \tfrac{\pi}{4}}{\tfrac{n \pi}{2} + \tfrac{\pi}{4}}, \quad 
	(\xi, \tau) = \mleft( \frac{x}{ \cos(2t) }, \frac{1}{2} \tan(2t) \mright) .
\end{equation}
This leads us to 
\begin{equation}
	\abs{x} = \abs{\xi} / ( 1 + 4 \tau^2 )^{1/2} , \quad dx \, dt = (1 + 4 \tau^2)^{-d/2 - 1} d\xi \, d\tau ,
\end{equation}
so that 
\begin{equation}
	\int_{(x, t) \in \R^d \times I_n} w(\abs{x}) \abs{  e^{- i t \Hamiltonian} u_0(x) }^2 \, dx \, dt 
	= \int_{(\xi, \tau) \in \R^d \times \R} \frac{w(\abs{\xi} / ( 1 + 4 \tau^2 )^{1/2})}{1 + 4\tau^2} \abs{ e^{- i \tau \Hamiltonianfree} u_0(\xi) }^2 \, d\xi \, d\tau 
\end{equation}
holds for each $n \in \Z$.
Summing this identity over $n \in \{0, 1, 2, 3\}$ and using the $2\pi$-periodicity of $e^{- i t \Hamiltonian}$, we obtain
\begin{equation} 
	\int_{(x, t) \in \R^d \times [0, 2 \pi]} w(\abs{x}) \abs{  e^{- i t \Hamiltonian} u_0(x) }^2 \, dx \, dt = 4 \int_{(\xi, \tau) \in \R^d \times \R} \frac{w(\abs{\xi} / ( 1 + 4 \tau^2 )^{1/2})}{1 + 4\tau^2} \abs{ e^{- i \tau \Hamiltonianfree} u_0(\xi) }^2 \, d\xi \, d\tau .
\end{equation}
Now replacing $u_0$ with $\psi(-\LambdaS) u_0$ gives \eqref{eq:relation between free and HO}, 
since the spherical Laplacian $\LambdaS$ commutes with both Hamiltonians $\Hamiltonian$ and $\Hamiltonianfree$.
\subsection{Smoothing estimates on the radial \texorpdfstring{$L^2$}{L2} space and its orthogonal complement} \label{section:radial initial data}
Let $\Lrad(\R^d)$ be the subspace of $L^2(\R^d)$ consisting of all radially symmetric functions, and $\Lperp(\R^d)$ be its orthogonal complement.
Then we have the following.
\begin{proposition} \label{prop:optimal constant radial}
Let $d \geq 2$.
We write $\hoConstwpdrad{w}{\psi}{d}$ and $\hoconstwpdrad{w}{\psi}{d}$ for the optimal constants of the inequalities \eqref{eq:smoothing estimate general} and \eqref{eq:smoothing estimate general reverse} with the initial data restricted to $\Lrad(\R^d)$, and $\hoConstwpdperp{w}{\psi}{d}$ and $\hoconstwpdperp{w}{\psi}{d}$ for those restricted to $\Lperp(\R^d)$, respectively.
Then we have
		\begin{alignat}{2}
	\hoConstwpdrad{w}{\psi}{d} &={}& 2\pi \sup_{j \in \N} {}&\psi(4j+d, (d/2-1)^2)^2 \Tjnuf{j}{d/2-1}{w} , 
	\label{eq:ho optimal constant C radial} \\
	\hoconstwpdrad{w}{\psi}{d} &={}&2\pi \inf_{j \in \N} {}&\psi(4j+d, (d/2-1)^2)^2 \Tjnuf{j}{d/2-1}{w} ,
	\label{eq:ho optimal constant c radial}
\end{alignat}
and
	\begin{alignat}{2}
	\hoConstwpdperp{w}{\psi}{d} &={}& 2\pi \sup_{j \in \N} \sup_{k \in \Z_{\geq 1}} {}&\psi(4j+2k+d, (k + d/2-1)^2)^2 \Tjnuf{j}{k+d/2-1}{w} , 
	\label{eq:ho optimal constant C perp} \\
	\hoconstwpdperp{w}{\psi}{d} &={}&2\pi \inf_{j \in \N} \inf_{k \in \Z_{\geq 1}} {}&\psi(4j+2k+d, (k + d/2-1)^2)^2 \Tjnuf{j}{k+d/2-1}{w} .
	\label{eq:ho optimal constant c perp}
\end{alignat}
\end{proposition}
This is an immediate consequence of the identity \eqref{eq:LHS} in the proof of Theorem \ref{thm:optimal constant} and the equivalence
\begin{equation}
f \in \mleft\{ \begin{matrix*}[r]
	\Lrad(\R^d) \\
	\Lperp(\R^d)
\end{matrix*} \mright. \iff \fjkn{f}{j}{k}{n} = 0 \text{ whenever } 
\mleft\{\begin{matrix*}[l]
k \geq 1 ,\\
k = 0,
\end{matrix*}
\mright.
\end{equation}
which is true for every $f \in L^2(\R^d)$.
Comparing Theorem \ref{thm:optimal constant} and Proposition \ref{prop:optimal constant radial}, we obtain
\begin{alignat}{4}
	\hoConstwpd{w}{\psi}{d} &= \max \{& \hoConstwpdrad{w}{\psi}{d} , {}& &\hoConstwpd{w}{\psi}{d+2} & \} = \sup_{k \in \N} {}&\hoConstwpdrad{w}{\psi}{d+2k} , 
	\label{eq:ho optimal constant C identity rad} \\
	\hoconstwpd{w}{\psi}{d} &= \min \{& \hoconstwpdrad{w}{\psi}{d} , {}& &\hoconstwpd{w}{\psi}{d+2} & \} = \inf_{k \in \N} {}&\hoconstwpdrad{w}{\psi}{d+2k} ,
	\label{eq:ho optimal constant c identity rad}
\end{alignat}
and
\begin{align}
\hoConstwpdperp{w}{\psi}{d} &= \hoConstwpd{w}{\psi}{d+2} , 
\label{eq:ho optimal constant C identity perp} \\
\hoconstwpdperp{w}{\psi}{d} &= \hoconstwpd{w}{\psi}{d+2} .
\label{eq:ho optimal constant c identity perp}
\end{align}
Using \eqref{eq:ho optimal constant C identity perp} and \eqref{eq:ho optimal constant c identity perp}, we get the following variants of Theorems \ref{main thm:type B HO} and \ref{main thm:type B identity HO} for the two-dimensional case. 
\begin{proposition}
The inequality
\begin{equation} \label{eq:type B HO 2D}
			\int_{(x, t) \in \R^2 \times \intervalcc{0}{2\pi}} \abs{x}^{-2} \abs{  e^{- i t \Hamiltonian} u_0(x) }^2 \, dx \, dt \leq 2\pi \norm{ u_0 }_{L^2(\R^2)}^2 
\end{equation}
holds for every $u_0 \in \Lperp(\R^2)$, where the constant $2\pi$ is optimal.
\end{proposition}
\begin{proposition}
	The identity
		\begin{equation} \label{eq:type B identity HO 2D}
	\int_{(x, t) \in \R^2 \times \intervalcc{0}{2\pi}} \abs{x}^{-2} \abs{ (-\Laplacian_{\S^{1}})^{1/4} e^{- i t \Hamiltonian} u_0(x) }^2 \, dx \, dt = 2\pi \norm{u_0}_{L^2(\R^2)}^2
\end{equation}
holds for every $u_0 \in \Lperp(\R^2)$.
\end{proposition}
We note that analogues of these for the free particle can be shown by a similar argument, or alternatively, by using the identity \eqref{eq:relation between free and HO}.  
\subsection{Theorem \ref{main thm:type C HO general} fails in the two-dimensional case} \label{section:fails in 2D}
Let $d=2$ and $w(r) = \weightphia{a}(r) \coloneqq (a + r^2)^{-1}$.
In this case, for each $a \in \intervaloo{0}{\infty}$, the optimal constant of the inequality \eqref{eq:type C HO general} satisfies
\begin{align}
	\frac{\Const}{4 \norm{\weightphia{a}}_{L^1(\intervaloo{0}{\infty})}} 
	&\geq \frac{2 \pi (4j+2)^{1/2} \Tjnuf{j}{0}{\weightphia{a}}}{2 \pi / a^{1/2}} \\
	&= (2 (2j+1) a)^{1/2} \int_{r \in \intervaloo{0}{\infty}} (a+r)^{-1} \exp(-r) \LaguerreL{j}{0}(r)^2 \, dr
\end{align}
for every $j \in \N$.
Now let $(a, j) = (1/6, 1)$. Then, by validated numerical integration, we obtain
\begin{equation}
	\frac{\Const}{4 \norm{\weightphia{a}}_{L^1(\intervaloo{0}{\infty})}} \geq \int_{r \in \intervaloo{0}{\infty}} (1/6+r)^{-1} \exp(-r) (r-1)^2 \, dr \geq 1.04350 .
\end{equation}
This lower bound can be verified by running the following code using
\texttt{kv} \cite{kv}.
\begin{lstlisting}[language=c++]
#include <kv/defint.hpp>
typedef kv::interval<double> itv;
struct Func {
    template <class T> T operator() (const T& x) {
        T y;
        y = (((6)*(exp(-(x))))*(pow((x)-(1),2)))/((1)+((6)*(x)));        
        return y;
    }
};
int main()
{
    std::cout.precision(17);
    std::cout << kv::defint_autostep(Func(), itv(0), itv(16), 6) << std::endl;
}
\end{lstlisting}
\subsection{Theorem \ref{main thm:type C HO general} fails for completely monotone functions} \label{section:fails for cm}
The assumption on $w$ cannot be relaxed to the complete monotonicity of $r \longmapsto w(r^{1/2})$.
To see this, let $w(r) = \gaussiana{a}(r) \coloneqq \exp(-a r^2)$.
In this case, we have
\begin{equation}
	\Tjnuf{0}{\nu}{\gaussiana{a}}
	= \frac{1}{\Gamma(\nu+1)} \int_{r \in \intervaloo{0}{\infty}} \exp(- (a+1) r) r^{\nu} \, dr
= (1+a)^{-(\nu+1)}
\end{equation}
for each $\nu \in \intervaloo{-1}{\infty}$ and $a \in \intervaloo{0}{\infty}$.
Therefore, for each $d \geq 2$ and $a \in \intervaloo{0}{\infty}$, the optimal constant $\Const$ of the inequality \eqref{eq:type C HO general} satisfies
\begin{equation}
\frac{\Const}{4 \norm{\gaussiana{a}}_{L^1(\intervaloo{0}{\infty})}} \geq \frac{2 \pi d^{1/2} \Tjnuf{0}{d/2 - 1}{\gaussiana{a}}}{2 (\pi / a)^{1/2}} = (\pi a d)^{1/2} (1+a)^{-d/2} .
\end{equation}
Now let $a = 1/(d-1)$. Then we get
\begin{equation}
\frac{\Const}{4 \norm{\gaussiana{a}}_{L^1(\intervaloo{0}{\infty})}} \geq \pi^{1/2} \mleft( 1 + \frac{1}{d-1} \mright)^{-(d-1)/2} > (\pi / e)^{1/2} \approx 1.07505 .
\end{equation}
\section*{Acknowledgment}
The author would like to thank Mitsuru Sugimoto (Nagoya University), Neal Bez (University of Tokyo), and the anonymous referee for pointing out that Theorems \ref{main thm:type B HO} and \ref{main thm:type B identity HO} can be derived from the corresponding results for free particles via the lens transform as in Section \ref{section:lens transform}.

% mystylenat standalone BBL; STY v3.11.
% Frozen display settings; no mystylenat.sty required.
\begingroup 
\ExplSyntaxOn 
\cs_set_protected:Npn\mystylenatBibliographySetup{
\cs_set:Npn \Journal ##1##2{##1}
\cs_set:Npn \Note ##1{}
\cs_set:Npn \ShowISBN ##1{}
\cs_set:Npn \ShowDOI ##1{##1}
\cs_set:Npn \ShowMR ##1{##1}
\cs_set:Npn \BookAddress ##1##2{##2}
\cs_set:Npn \Author ##1##2{##2}
\cs_set:Npn \Corrauthor ##1{##1}
\cs_if_exist:NT \__mystylenat_original_anchor:nn {\cs_set_eq:cN {NAT@anchor}\__mystylenat_original_anchor:nn }
}
\ExplSyntaxOff 

\endgroup 

\end{document}